\documentclass[a4paper,reqno,final,11pt]{amsart}
\usepackage{amssymb}
\usepackage{amsmath}
\usepackage{amsfonts}
\usepackage{hyperref}
\usepackage{graphicx}

\newtheorem{theorem}{Theorem}
\theoremstyle{plain}
\newtheorem{corollary}{Corollary}

\newtheorem{lemma}{Lemma}

\newtheorem{remark}{Remark}
\numberwithin{equation}{section}

\email{amsengouga@gmail.com,abdelmouhcene.sengouga@univ-msila.dz}
 
\begin{document}
\title[Boundary stabilization of a vibrating string]{Boundary stabilization of a vibrating string with variable length}

\author[S. Ghenimi]{Seyf Eddine Ghenimi}
\author[A. Sengouga]{Abdelmouhcene Sengouga}
\address[Seyf Eddine Ghenimi, Abdelmouhcene Sengouga]{ Laboratory of Functional Analysis and Geometry of Spaces\\
Department of mathematics\\
Faculty of Mathematics and Computer Sciences\\
University of M'sila\\
28000 M'sila, Algeria.}
\date{\today}
\maketitle

\begin{abstract}
We study small vibrations of a string with time-dependent length $\ell(t)$ and boundary damping. The vibrations are described by a 1-d wave equation in an interval with one moving endpoint at a speed $\ell'(t)$ slower than the speed of propagation of the wave c=1. With no damping, the energy of the solution decays if the interval is expanding and increases if the interval is shrinking. The energy decays faster when the interval is expanding and a constant damping is applied at the moving end. However, to ensure the energy decay in a shrinking interval, the damping factor $\eta$ must be close enough to the optimal value $\eta=1$, corresponding to the transparent condition. In all cases, we establish lower and upper estimates for the energy with explicit constants. 
\end{abstract}


\section{Introduction}

We consider small transversal vibrations of a uniform string, with a time
dependent length. The mechanical setting is sketched in Figure \ref{fig1}
where the left end of the string is fixed while the right moving end is also
allowed to move transversely and attached to a damping device (a dash-pot
with a damping factor $\eta \geq 0$).

\begin{figure}[tbph]
\centering\includegraphics[width=0.67\textwidth]{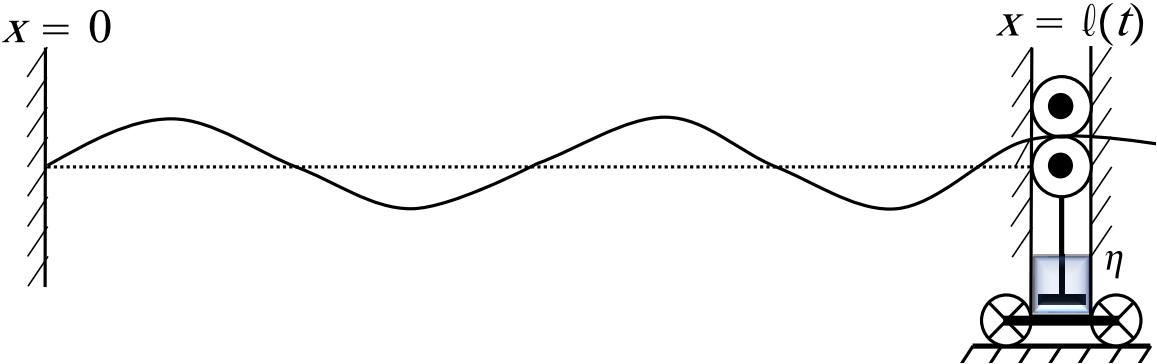}
\caption{A string with one moving end subject to a dash-pot damping.}
\label{fig1}
\end{figure}

Denoting the displacement function by $u$, depending on the position $x$
along the string and the time $t$, the model can be stated as follows 
\begin{equation}
\left\{ 
\begin{array}{ll}
u_{tt}-u_{xx}=0,\ \smallskip & \text{for }0<x<\ell (t)\text{ and }t>0, \\ 
\left( 1+\eta \ell ^{\prime }\left( t\right) \right) u_{x}\left( \ell \left(
t\right) ,t\right) +\left( \eta +\ell ^{\prime }\left( t\right) \right)
u_{t}\left( \ell \left( t\right) ,t\right) =0, \smallskip & 
\text{for }t>0, \\ 
u\left( 0,t\right) =0\text{,}\smallskip & \text{for }t>0, \\ 
u(x,0)=u^{0}\left( x\right) ,\text{ }u_{t}\left( x,0\right) =u^{1}\left(
x\right) ,\text{ \ } & \text{for }0<x<L.%
\end{array}%
\right.  \tag{WP}  \label{waveq}
\end{equation}%
The
subscripts $t$ and $x$ in (\ref{waveq}) stand for the derivatives in time and space variables
respectively. The functions $u^{0}$ and $u^{1}$ represents the initial shape
and the initial transverse speed of the string, respectively. The initial
length of the string is denoted by $L=\ell (0)$.

We assume that $\ell \in C\left( \left[ 0,+\infty \right[ \right) $ and that%
\begin{equation}
\left\vert \ell ^{\prime }\left( t\right) \right\vert <1,\ \ \ \text{for }%
t\geq 0,  \label{tlike}
\end{equation}%
which means that the speed of variation of the length of the string $\ell
^{\prime }\left( t\right) $ is strictly less then the speed of propagation
of the wave (here simplified to $c=1$).

In this paper, we are mainly interested in the asymptotic behaviour in time
of the energy of the solution, defined as%
\begin{equation}
E_{\ell }\left( t\right) :=\frac{1}{2}\int_{0}^{\mathbf{\ell }\left(
t\right) }u_{t}^{2}\left( x,t\right) +u_{x}^{2}\left( x,t\right) dx,\ \ \ 
\text{for }t\geq 0.  \label{E}
\end{equation}

For the time-independent interval, i.e. when $\ell \left( t\right) =L$ for $%
t\geq 0$, it is well known that:

\begin{itemize}
\item If $\eta =0$, then it is the energy is constant, i.e. $E_{L}\left(
t\right) =E_{L}\left( 0\right) ,$ for $t\geq 0$.

\item If $\eta \geq 0$ with $\eta \neq 1\ $, then the energy decays
exponentially%
\begin{equation*}
E_{L}\left( t\right) \leq CE_{L}\left( 0\right) e^{-\frac{1}{L}\ln
\left\vert \gamma _{\eta }\right\vert t},\text{ for some constant }C>0,
\end{equation*}%
where $\gamma _{\eta }:=\frac{1+\eta }{1-\eta }$, see \cite%
{Vese1988,CoZu1995,Cher1994,QuRu1977}. In \cite{GhSe2022a}, the present
authors showed that%
\begin{equation}
\frac{1}{\gamma _{\eta }^{2}}E_{L}\left( 0\right) e^{-\frac{1}{L}\ln
\left\vert \gamma _{\eta }\right\vert t}\leq E_{L}\left( t\right) \leq
\gamma _{\eta }^{2}E_{L}\left( 0\right) e^{-\frac{1}{L}\ln \left\vert \gamma
_{\eta }\right\vert t},\ \ \ \text{for }t\geq 0.
\end{equation}%
\end{itemize}

The question of $E_{\ell }\left( t\right)$ behaviour in time is more
delicate if the interval depends on time. For instance, for the Dirichlet
boundary conditions (which corresponds to $\eta =+\infty $ in (\ref{waveq}%
)), the energy decays if the interval is expanding and increases if the
interval is shrinking, see \cite{BaCh1981}. See also \cite%
{Seng2018,Seng2018a,SuLL2015} when the variation of the length is uniform in
time.

Regarding the case with a velocity feedback at the moving endpoint $x=\ell
\left( t\right) $:

\begin{itemize}
\item Gugat \cite{Guga2008} considered the case $u_{x}\left( \ell \left(
t\right) ,t\right) +cu_{t}\left( \ell \left( t\right) ,t\right) =0$ where $c$
is as constant. See also \cite{LuFe2020} for the particular $\ell (t)=1+vt$
where $v$ is a constant, $0<v<1$.

\item Ammari et al. \cite{AmBE2018} considered the case $u_{x}\left( \ell
\left( t\right) ,t\right) +f\left( t\right) u_{t}\left( \ell \left( t\right)
,t\right) =0,$ where $f\in L^{\infty }\left( 0,+\infty \right) \ $and $\ell
\left( t\right) $ is periodic. See also \cite{HaHo2019}\ where $\ell \left(
t\right) $ is not necessarily periodic. Mokhtari \cite{Mokh2022} considered
the case with two moving endpoints.
\end{itemize}

\emph{For the special case }$\eta =1$, the boundary condition at $x=\ell
\left( t\right) $ reads 
\begin{equation*}
u_{x}\left( \ell \left( t\right) ,t\right) +u_{t}\left( \ell \left( t\right)
,t\right) =0.
\end{equation*}%
This is a transparent condition, i.e. there is no reflections of waves from
the moving endpoint and consequently all the initial disturbances leave the
interval $\left( 0,\ell \left( t\right) \right) $ at most after a time 
\begin{equation*}
T_{\ell }:=\beta ^{-1}\left( L\right) ,
\end{equation*}%
where $\beta \left( t\right) :=t-\ell \left( t\right) ,$ see Figure \ref%
{fig2}. Hence, wether the interval is expanding or shrinking, the linear
velocity feedback $-u_{t}\left( \ell ^{\prime }\left( t\right) ,t\right) $
steers the solution to the zero state in the finite time $T_{\ell }$. See
for instance \cite{Guga2008,HaHo2019}, and for the particular case $\ell
\left( t\right) =L$ see \cite{Vese1988,CoZu1995}. In the remaining of this
paper, we will assume that $\eta \geq 0$ and $\eta \neq 1.$ 
\begin{figure}[tbph]
\centering\includegraphics[width=0.3\textwidth]{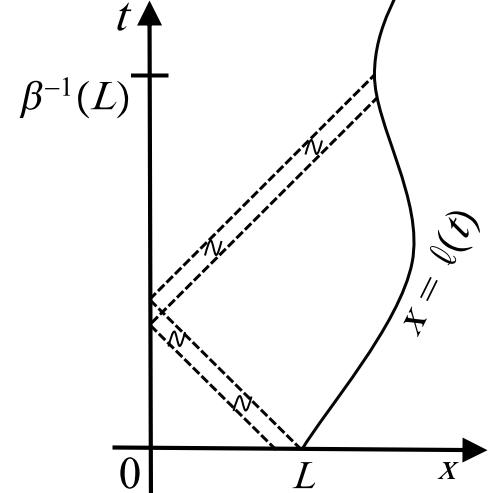}
\caption{An initial disturbance with a small support leaves the interval at
most after time $T_{\ell }$.}
\label{fig2}
\end{figure}

The approach used in the paper at hands is based on generalized Fourier
series. This is possible since the closed form for the solution of (\ref%
{waveq}) is given by the series%
\begin{equation}
u(x,t)=\sum_{n\in \mathbb{Z}}c_{n}\left( e^{\omega _{n}\varphi
(t+x)}-e^{\omega _{n}\varphi (t-x)}\right) ,\text{ \ \ \ for }0<x<\ell (t)%
\text{ and }t\geq 0,  \label{exact0}
\end{equation}%
where the coefficients $c_{n}$ can be explicitly computed in function of the
initial data $u^{0},u^{1}.\ $We denoted by $\omega _{n}$ a sequence of
complex number given by%
\begin{equation}
\omega _{n}:=-\frac{1}{2}\ln \left\vert \gamma _{\eta }\right\vert +\left\{ 
\begin{array}{ll}
\displaystyle\frac{2n+1}{2}i\pi , & \text{if }0\leq \eta <1,\smallskip \\ 
\displaystyle ni\pi , & \text{if }\eta >1,%
\end{array}%
\right.  \label{wn}
\end{equation}%
Observe that since $\left\vert \gamma _{\eta }\right\vert \geq 1$ for every $%
\eta \geq 0,$ the real part of $\omega _{n}$ is nonpositive.

By $\varphi ,$ we denoted a real function satisfying the functional equation 
\begin{equation}
\varphi \left( t+\ell \left( t\right) \right) -\varphi \left( t-\ell \left(
t\right) \right) =2.  \label{Moor}
\end{equation}
This equation often called Moor's equation following his paper \cite{Moor1970}. Some pairs of solutions $(\ell, \phi)$ can be found in \cite{VePo1975}.

We will
assume that $\varphi $ is differentiable and increasing. More precisely,%
\begin{equation}
\varphi \in C^{1}\left( \left[ -L,+\infty \right[ \right) \text{ \ and \ }%
\varphi ^{\prime }>0,\text{ \ for }t\geq -L.  \label{mono}
\end{equation}%
The assumption $\varphi ^{\prime }>0$ is needed in the sequel since $\varphi
^{\prime }$ will serve as a weight for an $L^2$ space. Besides, a
deacreasing $\phi$ can not satisfy (\ref{Moor}) since $\ell \left(
t\right)>0 $. See \cite{HaHo2019}
for further discussions on the regularity of the solutions of (\ref{Moor}).

In this work, we demonstrate how the series formulas (\ref{exact0}) can be
used to achieve the following results:

\begin{itemize}
\item \emph{For the undamped case, i.e. $\eta =0$}, the energy of the
solution satisfies%
\begin{equation}
\frac{m\left( t\right) }{M\left( t_{0}\right) }E_{\ell }\left( t_{0}\right)
\leq E_{\ell }\left( t\right) \leq \frac{M\left( t\right) }{m\left(
t_{0}\right) }E_{\ell }\left( t_{0}\right) ,\text{ \ \ \ for }0\leq t_{0}<t,
\label{estim1}
\end{equation}%
where%
\begin{equation}
m\left( t\right) :=\min_{x\in \left[ 0,\ell \left( t\right) \right] }\left\{
\varphi ^{\prime }(t-x),\varphi ^{\prime }(t+x)\right\} \ \text{and }M\left(
t\right) :=\max_{x\in \left[ 0,\ell \left( t\right) \right] }\left\{ \varphi
^{\prime }(t-x),\varphi ^{\prime }(t+x)\right\} .  \label{m.M}
\end{equation}%
See Theorem \ref{th0damp} and its corollaries for sharper estimates.

\item \emph{For the damped case $\eta >0,$ with $\eta \neq 1$}, the energy $%
E_{\ell }\left( t\right) $ satisfies%
\begin{multline}
\left( \frac{e^{\ln \left\vert \gamma _{\eta }\right\vert \varphi
(t_{0}-\ell \left( t_{0}\right) )}}{M\left( t_{0}\right) }\right) m\left(
t\right) e^{-\ln \left\vert \gamma _{\eta }\right\vert \varphi (t+\ell
\left( t\right) )}E_{\ell }\left( t_{0}\right) \leq E_{\ell }\left( t\right)
\label{exp-decay} \\
\leq \left( \frac{e^{\ln \left\vert \gamma _{\eta }\right\vert \varphi
(t_{0}+\ell \left( t_{0}\right) )}}{m\left( t_{0}\right) }\right) M\left(
t\right) e^{-\ln \left\vert \gamma _{\eta }\right\vert \varphi (t-\ell
\left( t\right) )}E_{\ell }\left( t_{0}\right) ,\text{ \ for }0\leq t_{0}<t.
\end{multline}%
See Theorem \ref{th-stab1} and it corollaries for more estimates. The
estimate given in (\ref{exp-decay}), with explicit constants, is new to the
best to our knowledge.
\end{itemize}

After the present introduction, we derive the exact solution and the\
expression for the coefficients of the series formula (\ref{exact0}). In
Section 3, we establish upper and lower estimates for the energy $E_{\ell
}(t)$ of the undamped equation. In Section 4, we deal with the damped case.
Some examples will be included as a last section.
\section{Exact solution}

Let us introduce the following family of Hilbert spaces%
\begin{equation*}
V_{0}\left( 0,\ell (t)\right) :=\left\{ w\in H^{1}\left( 0,\ell (t)\right) 
\text{, }w\left( 0\right) =0\right\} ,\ \ \ \text{\ for }t\geq 0
\end{equation*}%
and assume that the initial data satisfies%
\begin{equation}
u^{0}\in V_{0}\left( 0,L\right) ,\text{ \ }u^{1}\in L^{2}\left( 0,L\right) .
\label{ic}
\end{equation}

Let $T>0.$ Then, we have the following existence result for Problem (\ref%
{waveq}).

\begin{theorem}
\label{thexist1}Under the assumptions \emph{(\ref{tlike}), (\ref{mono}) }and 
\emph{(\ref{ic}), }Problem \emph{(\ref{waveq})} has a unique solution
satisfying%
\begin{equation}
u\in C\left( [0,T];V_{0}\left( 0,\ell (t)\right) \right) \cap C^{1}\left(
[0,T];L^{2}\left( 0,\ell (t)\right) \right) ,  \label{regul}
\end{equation}%
\emph{\ } given by the series \emph{(\ref{exact0})} where the coefficients $%
c_{n}\in 
\mathbb{C}
$ are computed as follows%
\begin{equation}
c_{n}=\frac{1}{4\omega _{n}}\int_{-L}^{L}\left( \tilde{u}_{x}^{0}+\tilde{u}%
^{1}\right) e^{-\omega _{n}\varphi (x)}dx\text{, \ \ \ for\ }n\in 
\mathbb{Z}
,\   \label{cn}
\end{equation}%
where\emph{\ }$\tilde{u}_{x}^{0}$ is an even $($resp. $\tilde{u}^{1}$ is an
odd$)$ extension of the initial data $u^{0}$ $($resp. $u^{1})$ defined on
the interval $\left( -L,L\right) $. Moreover,%
\begin{equation}
\sum_{n\in \mathbb{%
\mathbb{Z}
}}\left\vert \omega _{n}c_{n}\right\vert ^{2}=\frac{1}{8}\int_{-L}^{L}\left( 
\tilde{u}_{x}^{0}+\tilde{u}^{1}\right) ^{2}e^{\ln \left\vert \gamma _{\eta
}\right\vert \varphi (x)}\frac{dx}{\varphi ^{\prime }(x)}<+\infty .
\label{ncn}
\end{equation}
\end{theorem}

\begin{proof}
$\bullet $ \emph{The exact solution:} This part of solution is slightly
different from the approach in \cite{Vesn1971} where the author considered $1/\eta $ instead of $%
\eta $ in the boundary condition at $x=\ell(t)$. We include it here for the
sake clarity. The general solution of (\ref{waveq}) is given by D'Alembert's
formula 
\begin{equation}
u(x,t)=f(t+x)+g\left( t-x\right) ,  \label{dlmbrt}
\end{equation}%
where $f$ and $g$ are arbitrary continuous functions. The boundary
conditions at the endpoint $x=0$, we have%
\begin{equation*}
f(t)=-g\left( t\right) .
\end{equation*}%
The condition at $x=\ell \left( t\right) \ $implies that%
\begin{equation*}
\left( 1+\eta \ell ^{\prime }\left( t\right) \right) \left[ f^{\prime
}(t+\ell \left( t\right) )-g^{\prime }\left( t-\ell \left( t\right) \right) %
\right] =-\left( \eta +\ell ^{\prime }\left( t\right) \right) \left[
f^{\prime }(t+\ell \left( t\right) )+g^{\prime }\left( t-\ell \left(
t\right) \right) \right] ,
\end{equation*}%
hence%
\begin{equation}
\left[ 1+\eta \ell ^{\prime }\left( t\right) +\eta +\ell ^{\prime }\left(
t\right) \right] f^{\prime }\left( \alpha \left( t\right) \right) =-\left[
1+\eta \ell ^{\prime }\left( t\right) -\eta -\ell ^{\prime }\left( t\right) %
\right] f^{\prime }\left( \beta \left( t\right) \right) ,  \label{+a1}
\end{equation}%
where $\beta \left( t\right) :=t-\ell \left( t\right) $ and $\alpha \left(
t\right) :=t+\ell \left( t\right) $. Then, noting that 
\begin{equation}
\frac{1+\eta \ell ^{\prime }\left( t\right) -\eta -\ell ^{\prime }\left(
t\right) }{1+\eta \ell ^{\prime }\left( t\right) +\eta +\ell ^{\prime
}\left( t\right) }=\frac{1}{\gamma _{\eta }}\frac{\beta ^{\prime }\left(
t\right) }{\alpha ^{\prime }\left( t\right) },
\end{equation}%
we can rewrite (\ref{+a1}) as%
\begin{equation}
\alpha ^{\prime }\left( t\right) f^{\prime }\left( \alpha \left( t\right)
\right) =-\frac{1}{\gamma _{\eta }}\beta ^{\prime }\left( t\right) f^{\prime
}\left( \beta \left( t\right) \right) .
\end{equation}%
By integration, it follows that 
\begin{equation}
f\left( \alpha \left( t\right) \right) =-\frac{1}{\gamma _{\eta }}f\left(
\beta \left( t\right) \right) +C.  \label{a+b}
\end{equation}

Let us assume for the moment that $C=0.$ Then, it is convenient to search
for $f$ in the form $f(\xi )=e^{\omega \varphi \left( \xi \right) }$, for a
constant $\omega $ and some function $\varphi $. Substituting $e^{\omega
\varphi \left( \xi \right) }$\ in (\ref{a+b}), we get 
\begin{equation*}
e^{\omega \left[ \varphi \left( \alpha \left( t\right) \right) -\varphi
\left( \beta \left( t\right) \right) \right] }=-1/\gamma _{\eta }.
\end{equation*}%
Assuming that $\varphi $ satisfies (\ref{Moor}), we are led to the following
cases:

\begin{itemize}
\item[-] If $0\leq \eta <1$, then $\gamma _{\eta }\geq 1$ and we get 
\begin{equation*}
e^{\omega \left[ \varphi \left( \alpha \left( t\right) \right) -\varphi
\left( \beta \left( t\right) \right) \right] }=e^{\left( 2n+1\right) i\pi
-\ln \gamma _{\eta }}.
\end{equation*}%
Solving this equation for $\omega $, we obtain a sequence of values $\omega
_{n},n\in 
\mathbb{Z}
,$ where 
\begin{equation*}
\omega _{n}=\frac{2n+1}{2}i\pi -\frac{1}{2}\ln \gamma _{\eta }.
\end{equation*}

\item[-] If $\eta >1$, we have $\gamma _{\eta }<-1$ and we obtain this time 
\begin{equation*}
\omega _{n}=ni\pi -\frac{1}{2}\ln \left\vert \gamma _{\eta }\right\vert ,%
\text{ \ }n\in 
\mathbb{Z}
.
\end{equation*}
\end{itemize}

Thus, if $\eta \geq 0$ and $\eta \neq 1$, we always have $\ln \left\vert
\gamma _{\eta }\right\vert \geq 1$ and $\omega _{n}$ given by (\ref{wn}).

Due to the superposition principal, it follows that $f$ can be written as%
\begin{equation*}
f\left( \xi \right) =\sum_{n\in \mathbb{Z}}c_{n}e^{\omega _{n}\varphi \left(
\xi \right) },\text{ \ \ \ }c_{n}\in 
\mathbb{C}
,
\end{equation*}%
where $c_{n}$ are\ complex coefficients to be determined later. Since $%
f\left( \xi \right) =-g\left( \xi \right) ,$ then D'Alembert's formula for
the solution yields the series 
\begin{equation}
u(x,t)=\sum_{n\in \mathbb{Z}}c_{n}\left( e^{\omega _{n}\varphi \left(
t+x\right) }-e^{\omega _{n}\varphi \left( t-x\right) }\right) ,\text{ \ for }%
0<x<\ell (t)\text{ and }t\geq 0.  \label{phi_cn}
\end{equation}

If $C\neq 0$ in (\ref{a+b}), then we can check that 
\begin{equation*}
f\left( \xi \right) =\frac{C\gamma _{\eta }}{1+\gamma _{\eta }}+\sum_{n\in 
\mathbb{Z}}c_{n}e^{\omega _{n}\varphi \left( \xi \right) },
\end{equation*}%
solves (\ref{a+b}). However, this will not affect the solution of (\ref%
{waveq}) since $f\left( \xi \right) =-g\left( \xi \right) $ and thus the
constant parts of $f$ and $g$ will be cancelled in the expression (\ref%
{phi_cn}).

$\bullet $ \emph{Computing the coefficients }$c_{n}$\emph{:} We extend $%
u(\cdot ,t)$ to an odd function $\tilde{u}(\cdot ,t)$ on the interval $%
\left( -\ell \left( t\right) ,\ell \left( t\right) \right) .$ This ensures
in particular that the boundary condition at $x=0$ is satisfied for $t\geq
0. $ It follows also that $\tilde{u}_{t}(\cdot ,t)$ and $\tilde{u}_{x}(\cdot
,t) $ are respectively an odd and an even function. Going back to (\ref%
{phi_cn}), we infer that 
\begin{equation*}
\tilde{u}_{x}+\tilde{u}_{t}=2\varphi ^{\prime }(t+x)\sum_{n\in 
\mathbb{Z}
}\omega _{n}c_{n}e^{\omega _{n}\varphi (t+x)}\smallskip ,
\end{equation*}%
for $x\in \left( -\ell \left( t\right) ,\ell \left( t\right) \right) $ and $%
t\geq 0$. Using the definition of $\omega _{n},$ we get%
\begin{equation}
\tilde{u}_{x}+\tilde{u}_{t}=\left\{ 
\begin{array}{ll}
\displaystyle2e^{\frac{1}{2}\left( i\pi -\ln \gamma _{\eta }\right) \varphi
(t+x)}\varphi ^{\prime }(t+x)\sum_{n\in 
\mathbb{Z}
}\omega _{n}c_{n}e^{ni\pi \varphi (t+x)}\medskip , & \text{if }0\leq \eta <1,
\\ 
\displaystyle2e^{-\frac{1}{2}\ln \left\vert \gamma _{\eta }\right\vert
\varphi (t+x)}\varphi ^{\prime }(t+x)\sum_{n\in 
\mathbb{Z}
}\omega _{n}c_{n}e^{ni\pi \varphi (t+x)}, & \text{if }1<\eta <+\infty ,%
\end{array}%
\right.  \label{29}
\end{equation}%
which implies that 
\begin{equation}
\sum_{n\in 
\mathbb{Z}
}\omega _{n}c_{n}e^{ni\pi \varphi (t+x)}=\left\{ 
\begin{array}{ll}
\frac{1}{2\varphi ^{\prime }(t+x)}\displaystyle e^{\frac{1}{2}\left( -i\pi
+\ln \gamma _{\eta }\right) \varphi (t+x)}\left( \tilde{u}_{x}+\tilde{u}%
_{t}\right) \medskip , & \text{if }0\leq \eta <1, \\ 
\frac{1}{2\varphi ^{\prime }(t+x)}\displaystyle e^{\frac{1}{2}\ln \left\vert
\gamma _{\eta }\right\vert \varphi (t+x)}\left( \tilde{u}_{x}+\tilde{u}%
_{t}\right) , & \text{if }1<\eta <+\infty .%
\end{array}%
\right.  \label{30}
\end{equation}

Taking into account that $\left\{ e^{ni\pi \varphi (t+x)}/\sqrt{2}\right\}
_{n\in \mathbb{%
\mathbb{Z}
}}$ is an orthonormal basis of the weighted space $L^{2}\left( -\ell \left(
t\right) ,\ell \left( t\right) ,\varphi ^{\prime }\left( t+x\right)
dx\right) ,$ we deduce that%
\begin{equation*}
\omega _{n}c_{n}=\left\{ 
\begin{array}{ll}
\displaystyle\frac{1}{4}\int_{-\ell \left( t\right) }^{\ell \left( t\right)
}e^{-\frac{1}{2}\left( i\pi -\ln \gamma _{\eta }\right) \varphi (t+x)}\left( 
\tilde{u}_{x}+\tilde{u}_{t}\right) e^{-ni\pi \varphi (t+x)}dx\medskip , & 
\text{if }0\leq \eta <1, \\ 
\displaystyle\frac{1}{4}\int_{-\ell \left( t\right) }^{\ell \left( t\right)
}e^{\frac{1}{2}\ln \left\vert \gamma _{\eta }\right\vert \varphi
(t+x)}\left( \tilde{u}_{x}+\tilde{u}_{t}\right) e^{-ni\pi \varphi (t+x)}dx,
& \text{if }1<\eta <+\infty ,%
\end{array}%
\right.
\end{equation*}%
for\ $n\in 
\mathbb{Z}
$. Whether $0\leq \eta <1$ or $1<\eta <+\infty ,$ in both cases, we have%
\begin{equation}
c_{n}=\frac{1}{4\omega _{n}}\int_{-\ell \left( t\right) }^{\ell \left(
t\right) }\left( \tilde{u}_{x}+\tilde{u}_{t}\right) e^{-\omega _{n}\varphi
(t+x)}dx\text{, \ \ \ for\ }n\in 
\mathbb{Z}
\text{.}  \label{cn=}
\end{equation}

Taking $t=0$, we obtain (\ref{cn}) as claimed.

$\bullet $ \emph{Regularity of the solution:} As a consequence of Parseval's
equality, we get%
\begin{equation*}
\sum_{n\in 
\mathbb{Z}
}\left\vert \omega _{n}c_{n}\right\vert ^{2}=\left\{ 
\begin{array}{ll}
\displaystyle\frac{1}{8}\int_{-\ell \left( t\right) }^{\ell \left( t\right)
}\left\vert e^{-\frac{1}{2}\left( i\pi -\ln \gamma _{\eta }\right) \varphi
(t+x)}\right\vert ^{2}\left( \tilde{u}_{x}+\tilde{u}_{t}\right) ^{2}\frac{dx%
}{\varphi ^{\prime }(t+x)}\medskip , & \text{if }0\leq \eta <1 \\ 
\displaystyle\frac{1}{8}\int_{-\ell \left( t\right) }^{\ell \left( t\right)
}\left\vert e^{\frac{1}{2}\ln \left\vert \gamma _{\eta }\right\vert \varphi
(t+x)}\right\vert ^{2}\left( \tilde{u}_{x}+\tilde{u}_{t}\right) ^{2}\frac{dx%
}{\varphi ^{\prime }(t+x)}, & \text{if }1<\eta <+\infty .%
\end{array}%
\right.
\end{equation*}%
Whether $0\leq \eta <1$ or $1<\eta <+\infty ,$ the two cases of the
precedent identity can be written as%
\begin{equation}
\sum_{n\in 
\mathbb{Z}
}\left\vert \omega _{n}c_{n}\right\vert ^{2}=\frac{1}{8}\int_{-\ell \left(
t\right) }^{\ell \left( t\right) }\left( \tilde{u}_{x}+\tilde{u}_{t}\right)
^{2}e^{\ln \left\vert \gamma _{\eta }\right\vert \varphi (t+x)}\frac{dx}{%
\varphi ^{\prime }(t+x)}.  \label{ncn+}
\end{equation}%
Due to (\ref{mono}) and (\ref{ic}), we have for $t=0,$%
\begin{equation*}
\frac{1}{\sqrt{\varphi ^{\prime }(x)}}e^{\ln \left\vert \gamma _{\eta
}\right\vert \varphi (t+x)}\left( \tilde{u}_{x}^{0}+\tilde{u}^{1}\right) \in
L^{2}\left( -L,L\right)
\end{equation*}%
and (\ref{ncn}) follows.

Due to the continuity and differentiability of $\varphi \left( t+x\right) $
and the exponential function, and since $\left\vert w_{n}\right\vert
=O\left( n\right) $ for large values of $n$, the regularity result (\ref%
{regul}) follows from the convergences of the series of the solution (\ref%
{phi_cn}) and its derivatives.
\end{proof}

\section{The undamped case}

In this section, we show some results for the undamped case, i.e. $\eta =0$
in Problem (\ref{waveq}). To know wether the energy is increasing or in
creasing, we compute $E_{\ell }^{\prime }\left( t\right) .$ Thus, using
Leibniz's rule for differentiation under the integral sign, we get%
\begin{equation*}
E_{\ell }^{\prime }\left( t\right) =\frac{1}{2}\ell ^{\prime }(t)\left[
u_{x}^{2}\left( \ell (t),t\right) +u_{t}^{2}\left( \ell (t),t\right) \right]
+\int_{0}^{\ell \left( t\right) }u_{t}u_{tt}+u_{x}u_{tx}dx.
\end{equation*}%
Since $u_{tt}=u_{xx}$ , then $u_{t}u_{tt}+u_{x}u_{tx}=\left(
u_{t}u_{x}\right) _{x}\ $and it follows%
\begin{equation}
E_{\ell }^{\prime }\left( t\right) =\frac{1}{2}\ell ^{\prime }(t)\left[
u_{x}^{2}\left( \ell (t),t\right) +u_{t}^{2}\left( \ell (t),t\right) \right]
+u_{t}u_{x}\left( \ell (t),t\right) .  \label{dE}
\end{equation}
Then, we have the following result.

\begin{lemma}
The energy of solution of Problem \emph{(\ref{waveq})}, with $\eta =0$,
satisfies%
\begin{equation*}
E_{\ell }^{\prime }\left( t\right) =-\frac{\ell ^{\prime }(t)}{2}\left(
1-\left\vert \ell ^{\prime }\left( t\right) \right\vert ^{2}\right)
u_{t}^{2}\left( \ell \left( t\right) ,t\right) ,\ \ \text{for }t>0.
\end{equation*}
\end{lemma}

\begin{proof}
The boundary condition at $x=\ell \left( t\right) ,$ with $\eta =0,$ reads 
\begin{equation*}
u_{x}\left( \ell \left( t\right) ,t\right) =-u_{t}\left( \ell \left(
t\right) ,t\right) ,\ \ \text{for }t>0.
\end{equation*}%
Reporting this in (\ref{dE}), we get%
\begin{equation*}
E_{\ell }^{\prime }\left( t\right) =\frac{1}{2}\ell ^{\prime }(t)\left(
\left\vert \ell ^{\prime }\left( t\right) \right\vert ^{2}+1\right)
u_{t}^{2}\left( \ell \left( t\right) ,t\right) -\ell ^{\prime }\left(
t\right) u_{t}^{2}\left( \ell (t),t\right)
\end{equation*}%
and the lemma follows.
\end{proof}

\begin{remark}
The above lemma means that%
\begin{equation}
\Vert \text{ \ \ }E_{\ell }\left( t\right) \text{ is nondeacreasing if }%
-1<\ell ^{\prime }\left( t\right) \leq 0,\text{ and nonincreasing if }0\leq
\ell ^{\prime }\left( t\right) <1.  \label{E--ll}
\end{equation}%
The same result holds for the multidimensional wave equation, in a
time-dependent domain, with homogenous Dirichlet boundary conditions, see 
\cite{BaCh1981}.
\end{remark}

The next theorem show that the asymptotic behaviour of $E_{\ell }\left(
t\right) $ is dictated by $\varphi ^{\prime }.$

\begin{theorem}
\label{th0damp}Under the assumptions \emph{(\ref{tlike}), (\ref{mono})} and 
\emph{(\ref{ic}),} the solution of Problem \emph{(\ref{waveq})} satisfies%
\begin{multline}
\int_{0}^{\ell \left( t\right) }\left( \frac{1}{\varphi ^{\prime }(t+x)}+%
\frac{1}{\varphi ^{\prime }(t-x)}\right) \left( u_{x}^{2}+u_{t}^{2}\right) \\
+2\left( \frac{1}{\varphi ^{\prime }(t+x)}-\frac{1}{\varphi ^{\prime }(t-x)}%
\right) u_{x}u_{t}\ dx=4\mathcal{S}_{0},\text{ \ \ for }t\geq 0,
\label{E=4A}
\end{multline}%
\ \ where $\mathcal{S}_{0}:=\frac{\pi ^{2}}{2}\sum_{n\in 
\mathbb{Z}
}\left\vert \left( 2n+1\right) c_{n}\right\vert ^{2}$. Moreover, it holds
that%
\begin{equation}
\mathcal{S}_{0}m\left( t\right) \leq E_{\ell }\left( t\right) \leq \mathcal{S%
}_{0}M\left( t\right) ,\text{ \ \ for}\ t\geq 0,  \label{E0}
\end{equation}%
where $m\left( t\right) $ and $M\left( t\right) $ are defined in \emph{(\ref%
{m.M})}.
\end{theorem}

\begin{proof}
If $\eta =0$ then $\gamma _{\eta }=1$ and $\omega _{n}=\left( 2n+1\right)
i\pi /2$. The identity (\ref{ncn+}) becomes%
\begin{equation}
\frac{1}{8}\int_{-\ell \left( t\right) }^{\ell \left( t\right) }\left( 
\tilde{u}_{x}+\tilde{u}_{t}\right) ^{2}\frac{dx}{\varphi ^{\prime }(t+x)}=%
\frac{\pi ^{2}}{4}\sum_{n\in 
\mathbb{Z}
}\left\vert \left( 2n+1\right) c_{n}\right\vert ^{2}=\frac{1}{2}\mathcal{S}%
_{0},\text{ \ \ for}\ t\geq 0.  \label{eq-per+}
\end{equation}%
Since $\tilde{u}_{x}$ is an even function of $x$ and that $\tilde{u}_{t}$ is
an odd one, then changing $x$ by $-x$ in the last formula, we also obtain%
\begin{equation}
\int_{-\ell \left( t\right) }^{\ell \left( t\right) }\left( \tilde{u}_{x}-%
\tilde{u}_{t}\right) ^{2}\frac{dx}{\varphi ^{\prime }(t-x)}=4\mathcal{S}_{0},%
\text{ \ \ for}\ t\geq 0.  \label{eq-per-}
\end{equation}

Taking the sum of (\ref{eq-per+}) and (\ref{eq-per-}),\ we obtain 
\begin{equation*}
\int_{-\ell \left( t\right) }^{\ell \left( t\right) }\left( \tilde{u}_{x}+%
\tilde{u}_{t}\right) ^{2}\frac{dx}{\varphi ^{\prime }(t+x)}+\int_{-\ell
\left( t\right) }^{\ell \left( t\right) }\left( \tilde{u}_{x}-\tilde{u}%
_{t}\right) ^{2}\frac{dx}{\varphi ^{\prime }(t-x)}=8\mathcal{S}_{0}.
\end{equation*}%
Expanding $\left( u_{x}\pm u_{t}\right) ^{2}$ and collecting similar terms,
we get 
\begin{multline}
\int_{-\ell \left( t\right) }^{\ell \left( t\right) }\left( \frac{1}{\varphi
^{\prime }(t+x)}+\frac{1}{\varphi ^{\prime }(t-x)}\right) \left( \tilde{u}%
_{x}^{2}+\tilde{u}_{t}^{2}\right)  \label{est} \\
+2\left( \frac{1}{\varphi ^{\prime }(t+x)}-\frac{1}{\varphi ^{\prime }(t-x)}%
\right) \tilde{u}_{x}\tilde{u}_{t}\ dx=8\mathcal{S}_{0},\text{ \ \ for}\
t\geq 0.
\end{multline}%
As the function under the integral sign is even, then (\ref{E=4A}) follows.

Next, we use the algebraic inequality $\pm 2u_{x}u_{t}\leq
u_{t}^{2}+u_{x}^{2}$ to obtain%
\begin{multline*}
\int_{0}^{\ell \left( t\right) }\left( \frac{1}{\varphi ^{\prime }(t+x)}+%
\frac{1}{\varphi ^{\prime }(t-x)}-\left\vert \frac{1}{\varphi ^{\prime }(t+x)%
}-\frac{1}{\varphi ^{\prime }(t-x)}\right\vert \right) \left(
u_{x}^{2}+u_{t}^{2}\right) \ dx\leq 4\mathcal{S}_{0} \\
\leq \int_{0}^{\ell \left( t\right) }\left( \frac{1}{\varphi ^{\prime }(t+x)}%
+\frac{1}{\varphi ^{\prime }(t-x)}+\left\vert \frac{1}{\varphi ^{\prime
}(t+x)}-\frac{1}{\varphi ^{\prime }(t-x)}\right\vert \right) \left(
u_{x}^{2}+u_{t}^{2}\right) \ dx,
\end{multline*}%
for$\ t\geq 0.$ Recalling that 
\begin{equation}
\left( a+b\right) -\left\vert a-b\right\vert =2\min \left\{ a,b\right\} \ \
\ \text{and}\ \ \ \left( a+b\right) +\left\vert a-b\right\vert =2\max
\left\{ a,b\right\} ,  \label{max}
\end{equation}%
for $a,b\in 
\mathbb{R}
,$ then%
\begin{multline*}
\int_{0}^{\ell \left( t\right) }\min \left\{ \frac{1}{\varphi ^{\prime }(t+x)%
},\frac{1}{\varphi ^{\prime }(t-x)}\right\} \left(
u_{x}^{2}+u_{t}^{2}\right) \ dx\leq 2\mathcal{S}_{0} \\
\leq \int_{0}^{\ell \left( t\right) }\max \left\{ \frac{1}{\varphi ^{\prime
}(t+x)},\frac{1}{\varphi ^{\prime }(t-x)}\right\} \left(
u_{x}^{2}+u_{t}^{2}\right) \ dx,
\end{multline*}%
for$\ t\geq 0.$ Recalling that $E_{\ell }\left( t\right) $ is defined by (%
\ref{E}), we deduce that%
\begin{equation*}
\min_{x\in \left[ 0,\ell \left( t\right) \right] }\left\{ \frac{1}{\varphi
^{\prime }(t+x)},\frac{1}{\varphi ^{\prime }(t-x)}\right\} E_{\ell }\left(
t\right) \leq \mathcal{S}_{0}\leq \max_{x\in \left[ 0,\ell \left( t\right) %
\right] }\left\{ \frac{1}{\varphi ^{\prime }(t+x)},\frac{1}{\varphi ^{\prime
}(t-x)}\right\} E_{\ell }\left( t\right) ,
\end{equation*}%
hence%
\begin{equation*}
\min_{x\in \left[ 0,\ell \left( t\right) \right] }\left\{ \varphi ^{\prime
}(t+x),\varphi ^{\prime }(t-x)\right\} \mathcal{S}_{0}\leq E_{\ell }\left(
t\right) \leq \max_{x\in \left[ 0,\ell \left( t\right) \right] }\left\{
\varphi ^{\prime }(t+x),\varphi ^{\prime }(t-x)\right\} \mathcal{S}_{0},
\end{equation*}%
which is (\ref{E0}).
\end{proof}
\begin{remark} An estimation analogue to (\ref{E0}) was obtained for the case of homogeneous boundary conditions at both ends, see \cite{HaHo2019}.
\end{remark}
If $\varphi ^{\prime }$ is monotone then the asymptotic behaviour is
dictated by $\ell \left( t\right) $ and we have the following refinements.

\begin{corollary}
\label{Cor1}Under the assumption of Theorem \ref{th0damp}, assume that \ 
\begin{equation}
\varphi ^{\prime }\text{ is monotone for }t\in \lbrack t_{0},t_{1}],\text{ }%
0\leq t_{0}<t_{1}.  \label{monot}
\end{equation}
\end{corollary}

\begin{itemize}
\item If $-1<\ell ^{\prime }\left( t\right) \leq 0$ on $[t_{0},t_{1}]$, then 
$\varphi ^{\prime }$and $E_{\ell }\left( t\right) $ are nondeacreasing and%
\begin{equation}
\mathcal{S}_{0}\varphi ^{\prime }(t-\ell \left( t\right) )\leq E_{\ell
}\left( t\right) \leq \mathcal{S}_{0}\varphi ^{\prime }(t+\ell \left(
t\right) ),\text{ \ \ for}\ t\in \lbrack t_{0},t_{1}].  \label{ll-}
\end{equation}

\item If $0\leq \ell ^{\prime }\left( t\right) <1$ on $[t_{0},t_{1}]$, then $%
\varphi ^{\prime }$and $E_{\ell }\left( t\right) $ are nonincreasing and%
\begin{equation}
\mathcal{S}_{0}\varphi ^{\prime }(t+\ell \left( t\right) )\leq E_{\ell
}\left( t\right) \leq \mathcal{S}_{0}\varphi ^{\prime }\left( t-\ell \left(
t\right) \right) ,\text{ \ \ for}\ t\in \lbrack t_{0},t_{1}].  \label{ll+}
\end{equation}
\end{itemize}

\begin{proof}
We already have (\ref{E--ll}). The derivation of the identity (\ref{Moor})
yields 
\begin{equation}
\frac{1+\ell ^{\prime }\left( t\right) }{1-\ell ^{\prime }\left( t\right) }%
\varphi ^{\prime }(t+\ell \left( t\right) )=\varphi ^{\prime }(t-\ell \left(
t\right) ).  \label{+t-t}
\end{equation}%
Taking into consideration the variation of the function $s\mapsto \left(
1+s\right) /\left( 1-s\right) $ on the interval $\left( -1,1\right) $, it
follows that%
\begin{equation*}
0<\frac{1+\ell ^{\prime }\left( t\right) }{1-\ell ^{\prime }\left( t\right) }%
\leq 1\text{ \ if }-1<\ell ^{\prime }\left( t\right) \leq 0\ \ \ \text{and \
\ }1\leq \frac{1+\ell ^{\prime }\left( t\right) }{1-\ell ^{\prime }\left(
t\right) } \ \text{ \ if }0\leq \ell ^{\prime }\left( t\right) <1,
\end{equation*}%
hence 
\begin{gather}
\text{if \ }-1<\ell ^{\prime }\left( t\right) \leq 0,\text{ then \ }m\left(
t\right) =\varphi ^{\prime }(t-\ell \left( t\right) )\leq \varphi ^{\prime
}(t+\ell \left( t\right) )=M\left( t\right) ,  \label{ll--} \\
\text{if \ }0\leq \ell ^{\prime }\left( t\right) <1,\text{ then \ }M\left(
t\right) =\varphi ^{\prime }(t-\ell \left( t\right) )\geq \varphi ^{\prime
}(t+\ell \left( t\right) )=m\left( t\right) .  \label{ll++}
\end{gather}%
Of course, if $\varphi ^{\prime }$ is monotone and $-1<\ell ^{\prime }\left(
t\right) \leq 0$ then (\ref{ll--}) means that $\varphi ^{\prime }$ is
necessarily nondecreasing, and so is $\varphi ^{\prime }(t\pm \ell \left(
t\right) )\ $and $E_{\ell }\left( t\right) .$ The same argument can be made
when $0\leq \ell ^{\prime }\left( t\right) <1$. This shows the corollary.
\end{proof}

\begin{remark}
The assumption (\ref{monot}) is satisfied in all the examples of the last
section.
\end{remark}

Recall that $\mathcal{S}_{0}$ can be computed using the initial data by
setting $t=0$ in (\ref{E0}). If one needs to compare $E_{\ell }\left(
t\right) $ with $E_{\ell }\left( t_{0}\right) \ $for $0\leq t_{0}<t,$ then
we have the next result.

\begin{corollary}
Under the assumption of Theorem \ref{th0damp}, we have%
\begin{equation}
\frac{m\left( t\right) }{M\left( t_{0}\right) }E_{\ell }\left( t_{0}\right)
\leq E_{\ell }\left( t\right) \leq \frac{M\left( t\right) }{m\left(
t_{0}\right) }E_{\ell }\left( t_{0}\right) ,\text{ \ \ \ for }0\leq t_{0}<t%
\text{.}  \label{Et0}
\end{equation}%
Moreover, if $\varphi $ satisfies (\ref{monot}), then:
\end{corollary}

\begin{itemize}
\item If $-1<\ell ^{\prime }\left( t\right) \leq 0$ on $[t_{0},t_{1}]$, then 
$\varphi ^{\prime }$and $E_{\ell }\left( t\right) $ are nondecreasing and
satisfy%
\begin{equation}
\frac{\varphi ^{\prime }\left( t-\ell \left( t\right) \right) }{\varphi
^{\prime }\left( t_{0}+\ell \left( t_{0}\right) \right) }E_{\ell }\left(
t_{0}\right) \leq E_{\ell }\left( t\right) \leq \frac{\varphi ^{\prime
}\left( t+\ell \left( t\right) \right) }{\varphi ^{\prime }\left( t_{0}-\ell
\left( t_{0}\right) \right) }E_{\ell }\left( t_{0}\right) ,\text{ \ \ for }%
t\in \lbrack t_{0},t_{1}].  \label{Et0-}
\end{equation}

\item If $0\leq \ell ^{\prime }\left( t\right) <1$ on $[t_{0},t_{1}]$, then $%
\varphi ^{\prime }$and $E_{\ell }\left( t\right) $ are nonincreasing and%
\begin{equation}
\frac{\varphi ^{\prime }\left( t+\ell \left( t\right) \right) }{\varphi
^{\prime }\left( t_{0}-\ell \left( t_{0}\right) \right) }E_{\ell }\left(
t_{0}\right) \leq E_{\ell }\left( t\right) \leq \frac{\varphi ^{\prime
}\left( t-\ell \left( t\right) \right) }{\varphi ^{\prime }\left( t_{0}+\ell
\left( t_{0}\right) \right) }E_{\ell }\left( t_{0}\right) ,\text{ \ \ for }%
t\in \lbrack t_{0},t_{1}].  \label{Et0+}
\end{equation}
\end{itemize}

\begin{proof}
Since (\ref{E0}) holds also for $t=t_{0}$, then the corollary follows from
the inequalities%
\begin{equation*}
\frac{E_{\ell }\left( t\right) }{M\left( t\right) }\leq \mathcal{S}_{0}\leq 
\frac{E_{\ell }\left( t_{0}\right) }{m\left( t_{0}\right) }\text{ \ \ and \
\ }\frac{E_{\ell }\left( t_{0}\right) }{M\left( t_{0}\right) }\leq \mathcal{S%
}_{0}\leq \frac{E_{\ell }\left( t\right) }{m\left( t\right) }.
\end{equation*}
\end{proof}

\section{The damped case}

In this section, we investigate the case with boundary damping when%
\begin{equation*}
\eta >0\text{ \ and \ }\eta \neq 1
\end{equation*}%
in Problem (\ref{waveq}). Let us recall that we still have (\ref{dE}), i.e. 
\begin{equation*}
E_{\ell }^{\prime }\left( t\right) =\frac{1}{2}\ell ^{\prime }(t)\left[
u_{x}^{2}\left( \ell (t),t\right) +u_{t}^{2}\left( \ell (t),t\right) \right]
+u_{t}u_{x}\left( \ell (t),t\right) ,\ \ \text{\ for}\ t\geq 0.
\end{equation*}%
but now the boundary condition at $x=\ell \left( t\right) $ is 
\begin{equation*}
\left( 1+\eta \ell ^{\prime }\left( t\right) \right) u_{x}\left( \ell \left(
t\right) ,t\right) +\left( \eta +\ell ^{\prime }\left( t\right) \right)
u_{t}\left( \ell \left( t\right) ,t\right) =0,\ \ \text{\ for}\ t\geq 0
\end{equation*}

Let us discuss the sign of $E_{\ell }^{\prime }\left( t\right) $ for
different values of $\ell ^{\prime }\left( t\right) .$

\begin{itemize}
\item If the interval is shrinking $-1<\ell ^{\prime }\left( t\right) <0,$
for $t\in \left[ t_{1},t_{2}\right] \ $where $0\leq t_{1}<t_{2},$ then we
have the following cases:

- If $\ell ^{\prime }\left( t\right) =-1/\eta ,$ then the boundary condition
at $x=\ell \left( t\right) $ reads $u_{t}\left( \ell (t),t\right) =0$ and
thus 
\begin{equation*}
E_{\ell }^{\prime }\left( t\right) =-\frac{1}{2\eta }u_{x}^{2}\left( \ell
(t),t\right) \leq 0,\text{ \ for }t\in \left[ t_{1},t_{2}\right] .
\end{equation*}

- If $\ell ^{\prime }\left( t\right) =-\eta ,$ then $u_{x}\left( \ell
(t),t\right) =0$ and thus 
\begin{equation*}
E_{\ell }^{\prime }\left( t\right) =-\frac{1}{2}\eta u_{t}^{2}\left( \ell
(t),t\right) \leq 0,\text{ \ for }t\in \left[ t_{1},t_{2}\right] .
\end{equation*}

- Assume that $\ell ^{\prime }\left( t\right) \notin \left\{ -1/\eta
,0,-\eta \right\} ,$ then after some computation we can rewrite (\ref{dE}) as%
\begin{equation}
E_{\ell }^{\prime }\left( t\right) =-\frac{1}{2}\left( \ell ^{\prime }\left(
t\right) \eta ^{2}+2\eta +\ell ^{\prime }\left( t\right) \right) \frac{%
1-\left\vert \ell ^{\prime }\left( t\right) \right\vert ^{2}}{\left( 1+\eta
\ell ^{\prime }\left( t\right) \right) ^{2}}u_{t}^{2}\left( \ell \left(
t\right) ,t\right) ,\text{ \ for }t\in \left[ t_{1},t_{2}\right] .
\label{E'2}
\end{equation}%
The sign of $E_{\ell }^{\prime }\left( t\right) $ is opposite to the sign of 
\begin{equation*}
P_{\ell }\left( \eta \right) =\ell ^{\prime }\left( t\right) \eta ^{2}+2\eta
+\ell ^{\prime }\left( t\right) .
\end{equation*}%
Due to (\ref{tlike}), the polynomial $P_{\ell }\left( \eta \right) $ has a
discriminant $\Delta =4(1-\left\vert \ell ^{\prime }\left( t\right)
\right\vert ^{2})>0$ and thus $P_{\ell }\left( \eta \right) $ has two real
roots 
\begin{equation}
\eta _{1}:=\left( -1+\sqrt{1-\left\vert \ell ^{\prime }\left( t\right)
\right\vert ^{2}}\right) /\ell ^{\prime }\left( t\right) \text{ \ and \ }%
\eta _{2}:=\left( -1-\sqrt{1-\left\vert \ell ^{\prime }\left( t\right)
\right\vert ^{2}}\right) /\ell ^{\prime }\left( t\right) .  \label{eta1,2}
\end{equation}

\begin{figure}[tbph]
\centering\includegraphics[width=0.45\textwidth]{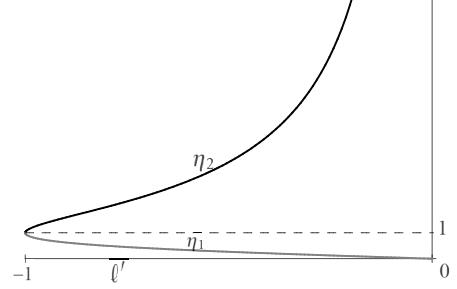}
\caption{Variation of $\protect\eta_1 $ and $\protect\eta_2$ in function of $%
\ell ^{\prime }$ when $-1<\ell ^{\prime } <0$.}
\label{fig4}
\end{figure}

Both $\eta _{1}$ and $\eta _{2}$ have the opposite sign of $\ell ^{\prime
}\left( t\right) ,$ for $t\in \left[ t_{1},t_{2}\right] $ and in particular%
\begin{equation*}
0<\eta _{1}<1<\eta _{2}.
\end{equation*}%
We deduce from (\ref{E'2}) that:

\begin{itemize}
\item If $\eta _{1}<\eta <\eta _{2},$ for $t\in \left[ t_{1},t_{2}\right] ,$
then $P_{\ell }\left( \eta \right) >0$ and by consequence $E_{\ell }^{\prime
}\left( t\right) <0.$\label{shrink+decay}

\item If $\eta =\eta _{1}$ or $\eta =\eta _{2},$ for $t\in \left[ t_{1},t_{2}%
\right] ,$ then $P_{\ell }\left( \eta \right) =E_{\ell }^{\prime }\left(
t\right) =0,$ i.e. the energy is constant.

\item If $\eta \in \left]0, \eta _{1}\right[ \cup \left] \eta _{2},\infty%
\right[ ,$ for $t\in \left[ t_{1},t_{2}\right] ,$ then $P_{\ell }\left( \eta
\right) <0$ and we have $E_{\ell }^{\prime }\left( t\right) >0.$
\end{itemize}

\item If the interval is independent of time, i.e. $\ell ^{\prime }\left(
t\right) =0\ $for $t\in \left[ t_{1},t_{2}\right] ,$ then $u_{x}\left( \ell
\left( t\right) ,t\right) =-\eta u_{t}\left( \ell \left( t\right) ,t\right) $
and thus%
\begin{equation*}
E_{\ell }^{\prime }\left( t\right) =-\eta u_{t}^{2}\left( \ell \left(
t\right) ,t\right) \leq 0,\ \ \text{for }t\in \left[ t_{1},t_{2}\right] .
\end{equation*}

\item If the interval is expanding, i.e. $0<\ell ^{\prime }\left( t\right)
<1\ $for $t\in \left[ t_{1},t_{2}\right] ,$ then $P_{\ell }\left( \eta
\right) >0$ for $\eta >0.$ Taking into account (\ref{tlike}) and (\ref{E'2}%
), we deduce that $E_{\ell }^{\prime }\left( t\right) \leq 0,$ \ for $t\in %
\left[ t_{1},t_{2}\right] .$
\end{itemize}

\begin{remark}
\label{rmkEd}To summarize, under the assumption $(\ref{tlike}),$%
\begin{equation}
\left\Vert \text{\ \ \ \ }%
\begin{array}{l}
\text{If the interval is expanding, then }E_{\ell }\left( t\right) \text{ is
nonincreasing for any }\eta >0\text{.} \\ 
\text{If the interval is shrinking, then }E_{\ell }\left( t\right) \text{ is
nonincreasing if the damping} \\ 
\text{factor can be taken close enough to the optimal value }\eta =1.%
\end{array}%
\right.  \label{E3}
\end{equation}
\end{remark}

Let us now estimate $E_{\ell }\left( t\right) $ is using $\varphi ^{\prime }$
and $\exp \left( -\ln \left\vert \gamma _{\eta }\right\vert \varphi \right) $%
.

\begin{theorem}
\label{th-stab1}Under the assumptions \emph{(\ref{tlike}), (\ref{mono})} and 
\emph{(\ref{ic}),} the solution of Problem \emph{(\ref{waveq})} satisfies%
\begin{multline}
\int_{0}^{\ell \left( t\right) }\left( \frac{e^{\ln \left\vert \gamma _{\eta
}\right\vert \varphi (t+x)}}{\varphi ^{\prime }(t+x)}+\frac{e^{\ln
\left\vert \gamma _{\eta }\right\vert \varphi (t-x)}}{\varphi ^{\prime }(t-x)%
}\right) \left( u_{x}^{2}+u_{t}^{2}\right) \\
+2\left( \frac{e^{\ln \left\vert \gamma _{\eta }\right\vert \varphi (t+x)}}{%
\varphi ^{\prime }(t+x)}-\frac{e^{\ln \left\vert \gamma _{\eta }\right\vert
\varphi (t-x)}}{\varphi ^{\prime }(t-x)}\right) u_{x}u_{t}dx=4\mathcal{S}%
_{\eta },  \label{E.D}
\end{multline}%
\ \ for$\ t\geq 0$, where $\mathcal{S}_{\eta }:=2\sum_{n\in 
\mathbb{Z}
}\left\vert \omega _{n}c_{n}\right\vert ^{2}$. Moreover, it holds that%
\begin{equation}
\mathcal{S}_{\eta }\tilde{m}\left( t\right) \leq E_{\ell }\left( t\right)
\leq \mathcal{S}_{\eta }\tilde{M}\left( t\right) ,\text{ \ \ for}\ t\geq 0,
\label{stab}
\end{equation}%
where 
\begin{eqnarray*}
\tilde{m}\left( t\right) &:&=\min_{x\in \left[ 0,\ell \left( t\right) \right]
}\left\{ \varphi ^{\prime }(t-x)e^{-\ln \left\vert \gamma _{\eta
}\right\vert \varphi (t-x)},\varphi ^{\prime }(t+x)e^{-\ln \left\vert \gamma
_{\eta }\right\vert \varphi (t+x)}\right\} , \\
\tilde{M}\left( t\right) &:&=\max_{x\in \left[ 0,\ell \left( t\right) \right]
}\left\{ \varphi ^{\prime }(t-x)e^{-\ln \left\vert \gamma _{\eta
}\right\vert \varphi (t-x)},\varphi ^{\prime }(t+x)e^{-\ln \left\vert \gamma
_{\eta }\right\vert \varphi (t+x)}\right\} .
\end{eqnarray*}
\end{theorem}

\begin{proof}
We argue as in the proof of Theorem \ref{th0damp}. Noting that $\tilde{u}%
_{x} $ is an even function of $x$ and that $\tilde{u}_{t}$ is an odd one,
then the identity (\ref{ncn+}) yields%
\begin{equation}
\int_{-\ell \left( t\right) }^{\ell \left( t\right) }e^{\ln \left\vert
\gamma _{\eta }\right\vert \varphi (t\pm x)}\left( \tilde{u}_{x}\pm \tilde{u}%
_{t}\right) ^{2}\frac{dx}{\varphi ^{\prime }(t\pm x)}=4\mathcal{S}_{\eta },%
\text{ \ \ for}\ t\geq 0.  \label{32}
\end{equation}%
Hence, summing, we get 
\begin{equation*}
\int_{-\ell \left( t\right) }^{\ell \left( t\right) }e^{\ln \left\vert
\gamma _{\eta }\right\vert \varphi (t+x)}\left( \tilde{u}_{x}+\tilde{u}%
_{t}\right) ^{2}\frac{dx}{\varphi ^{\prime }(t+x)}+\int_{-\ell \left(
t\right) }^{\ell \left( t\right) }e^{\ln \left\vert \gamma _{\eta
}\right\vert \varphi (t-x)}\left( \tilde{u}_{x}-\tilde{u}_{t}\right) ^{2}%
\frac{dx}{\varphi ^{\prime }(t-x)}=8\mathcal{S}_{\eta }.
\end{equation*}%
Expanding squares, we get%
\begin{multline}
\int_{-\ell \left( t\right) }^{\ell \left( t\right) }\left( \frac{e^{\ln
\left\vert \gamma _{\eta }\right\vert \varphi (t+x)}}{\varphi ^{\prime }(t+x)%
}+\frac{e^{\ln \left\vert \gamma _{\eta }\right\vert \varphi (t-x)}}{\varphi
^{\prime }(t-x)}\right) \left( \tilde{u}_{x}^{2}+\tilde{u}_{t}^{2}\right)
\label{=E0} \\
+2\left( \frac{e^{\ln \left\vert \gamma _{\eta }\right\vert \varphi (t+x)}}{%
\varphi ^{\prime }(t+x)}-\frac{e^{\ln \left\vert \gamma _{\eta }\right\vert
\varphi (t-x)}}{\varphi ^{\prime }(t-x)}\right) \tilde{u}_{x}\tilde{u}_{t}%
\text{ }dx=8\mathcal{S}_{\eta },\text{ \ \ for}\ t\geq 0.
\end{multline}%
As the function under the integral sign is even, then (\ref{E.D}) follows.

For $0\leq x\leq \ell \left( t\right) $ and $t\geq 0,$ let us denote 
\begin{equation*}
A\left( x,t\right) =\frac{e^{\ln \left\vert \gamma _{\eta }\right\vert
\varphi (t-x)}}{\varphi ^{\prime }(t-x)}\text{ and \ }B\left( x,t\right) =%
\frac{e^{\ln \left\vert \gamma _{\eta }\right\vert \varphi (t+x)}}{\varphi
^{\prime }(t+x)}.
\end{equation*}%
Then, we can rewrite (\ref{E.D}) as 
\begin{equation*}
\int_{0}^{\ell \left( t\right) }\left( A\left( x,t\right) +B\left(
x,t\right) \right) \left( \tilde{u}_{t}^{2}+\tilde{u}_{x}^{2}\right)
dx+2\left( A\left( x,t\right) -B\left( x,t\right) \right) \tilde{u}_{t}%
\tilde{u}_{x}dx=4\mathcal{S}_{\eta }.
\end{equation*}%
Using the algebraic inequality%
\begin{equation*}
-\left\vert A-B\right\vert \left( u_{t}^{2}+u_{x}^{2}\right) \leq 2\left(
A-B\right) u_{t}u_{x}\leq \left\vert A-B\right\vert \left(
u_{t}^{2}+u_{x}^{2}\right) ,
\end{equation*}%
we get%
\begin{equation*}
\int_{0}^{\ell \left( t\right) }\left( \left( A+B\right) -\left\vert
A-B\right\vert \right) \left( \tilde{u}_{t}^{2}+\tilde{u}_{x}^{2}\right)
dx\leq 4\mathcal{S}_{\eta }\leq \int_{0}^{\ell \left( t\right) }\left(
\left( A+B\right) +\left\vert A-B\right\vert \right) \left( \tilde{u}%
_{t}^{2}+\tilde{u}_{x}^{2}\right) dx.
\end{equation*}%
Thanks to (\ref{max}), the precedent estimation yields%
\begin{multline*}
\int_{0}^{\ell \left( t\right) }\min \left\{ \frac{e^{\ln \left\vert \gamma
_{\eta }\right\vert \varphi (t-x)}}{\varphi ^{\prime }(t-x)},\frac{e^{\ln
\left\vert \gamma _{\eta }\right\vert \varphi (t+x)}}{\varphi ^{\prime }(t+x)%
}\right\} \left( \tilde{u}_{t}^{2}+\tilde{u}_{x}^{2}\right) dx\leq 2\mathcal{%
S}_{\eta } \\
\leq \int_{0}^{\ell \left( t\right) }\max \left\{ \frac{e^{\ln \left\vert
\gamma _{\eta }\right\vert \varphi (t-x)}}{\varphi ^{\prime }(t-x)},\frac{%
e^{\ln \left\vert \gamma _{\eta }\right\vert \varphi (t+x)}}{\varphi
^{\prime }(t+x)}\right\} \left( \tilde{u}_{t}^{2}+\tilde{u}_{x}^{2}\right) dx%
\text{,}
\end{multline*}%
\ for $t\geq 0.$ By consequence 
\begin{multline*}
\min_{x\in \left[ 0,\ell \left( t\right) \right] }\left\{ \frac{e^{\ln
\left\vert \gamma _{\eta }\right\vert \varphi (t-x)}}{\varphi ^{\prime }(t-x)%
},\frac{e^{\ln \left\vert \gamma _{\eta }\right\vert \varphi (t+x)}}{\varphi
^{\prime }(t+x)}\right\} E_{\ell }\left( t\right) \leq \mathcal{S}_{\eta } \\
\leq \max_{x\in \left[ 0,\ell \left( t\right) \right] }\left\{ \frac{e^{\ln
\left\vert \gamma _{\eta }\right\vert \varphi (t-x)}}{\varphi ^{\prime }(t-x)%
},\frac{e^{\ln \left\vert \gamma _{\eta }\right\vert \varphi (t+x)}}{\varphi
^{\prime }(t+x)}\right\} E_{\ell }\left( t\right) .
\end{multline*}

This implies (\ref{stab}).
\end{proof}

Since $\ln \left\vert \gamma _{\eta }\right\vert \geq 0 $ for $\eta \geq 0,$
and $\varphi $ is nondecreasing, we have the following immediate corollary.

\begin{corollary}
Under the assumption of Theorem \ref{th-stab1}, it holds that%
\begin{equation}
\mathcal{S}_{\eta }m\left( t\right) e^{-\ln \left\vert \gamma _{\eta
}\right\vert \varphi (t+\ell \left( t\right) )}\leq E_{\ell }\left( t\right)
\leq \mathcal{S}_{\eta }M\left( t\right) e^{-\ln \left\vert \gamma _{\eta
}\right\vert \varphi (t-\ell \left( t\right) )},\text{ \ \ for}\ t\geq 0,
\label{stab0}
\end{equation}%
$m\left( t\right) $ and $M\left( t\right) $ are given by \emph{(\ref{m.M})}.
\end{corollary}

If $\varphi ^{\prime }$ is monotone, then (\ref{stab0}) can be replaced by
more explicit estimation.

\begin{corollary}
Under the assumption of Theorem \ref{th-stab1}, assume that $\varphi $
satisfies (\ref{monot}) on $[t_{0},t_{1}]$, then:
\end{corollary}

\begin{itemize}
\item If $-1<\ell ^{\prime }\left( t\right) \leq 0$ on $[t_{0},t_{1}]$, then 
$\varphi ^{\prime }$ is nondecreasing and $E_{\ell }\left( t\right) $
satisfies 
\begin{equation}
\mathcal{S}_{\eta }\varphi ^{\prime }(t-\ell \left( t\right) )e^{-\ln
\left\vert \gamma _{\eta }\right\vert \varphi (t+\ell \left( t\right) )}\leq
E_{\ell }\left( t\right) \leq \mathcal{S}_{\eta }\varphi ^{\prime }\left(
t+\ell \left( t\right) \right) e^{-\ln \left\vert \gamma _{\eta }\right\vert
\varphi (t-\ell \left( t\right) )},\text{ \ \ for }t\in \lbrack t_{0},t_{1}].
\label{stab1.}
\end{equation}

\item If $0<\ell ^{\prime }\left( t\right) <1$ on $[t_{0},t_{1}]$, then $%
\varphi ^{\prime }$and $E_{\ell }\left( t\right) $ are nonincreasing and%
\begin{equation}
\mathcal{S}_{\eta }\varphi ^{\prime }(t+\ell \left( t\right) )e^{-\ln
\left\vert \gamma _{\eta }\right\vert \varphi (t+\ell \left( t\right) )}\leq
E_{\ell }\left( t\right) \leq \mathcal{S}_{\eta }\varphi ^{\prime }\left(
t-\ell \left( t\right) \right) e^{-\ln \left\vert \gamma _{\eta }\right\vert
\varphi (t-\ell \left( t\right) )},\text{ \ \ for }t\in \lbrack t_{0},t_{1}].
\label{stab2.}
\end{equation}
\end{itemize}

\begin{proof}
It suffices to argue as in the proof of Corollary \ref{Cor1}.
\end{proof}

\begin{remark}
If the interval is shrinking, there is competition between the nondecreasing 
$\varphi ^{\prime }$ and $e^{-\ln \left\vert \gamma _{\eta }\right\vert
\varphi }$ in estimation (\ref{stab1.}). The behaviour of $E_{\ell }\left(
t\right) $ depends on the value of the damping $\eta $ as stated in Remark %
\ref{rmkEd}.
\end{remark}

To compare $E_{\ell }\left( t\right) $ with the energy $E_{\ell }\left(
t_{0}\right) $ for $0\leq t_{0}<t,$ we have the following result.

\begin{corollary}
\label{CoroEt0}Under the assumption of Theorem \ref{th-stab1}, we have 
\begin{multline}
\frac{m\left( t\right) e^{-\ln \left\vert \gamma _{\eta }\right\vert \varphi
(t+\ell \left( t\right) )}}{M\left( t_{0}\right) e^{-\ln \left\vert \gamma
_{\eta }\right\vert \varphi (t_{0}-\ell \left( t_{0}\right) )}}E_{\ell
}\left( t_{0}\right) \leq E_{\ell }\left( t\right)  \label{dEt0} \\
\leq \frac{M\left( t\right) e^{-\ln \left\vert \gamma _{\eta }\right\vert
\varphi (t-\ell \left( t\right) )}}{m\left( t_{0}\right) e^{-\ln \left\vert
\gamma _{\eta }\right\vert \varphi (t_{0}+\ell \left( t_{0}\right) )}}%
E_{\ell }\left( t_{0}\right) ,\text{ for }0\leq t_{0}<t.
\end{multline}%
Moreover, if $\varphi ^{\prime }$ satisfies (\ref{monot}) on $[t_{0},t_{1}]$%
, then:
\end{corollary}

\begin{itemize}
\item If $-1<\ell ^{\prime }\left( t\right) \leq 0$ on $[t_{0},t_{1}]$, then 
$\varphi ^{\prime }$ is nondecreasing and $E_{\ell }\left( t\right) $
satisfies%
\begin{multline}
\frac{\varphi ^{\prime }\left( t-\ell \left( t\right) \right) e^{-\ln
\left\vert \gamma _{\eta }\right\vert \varphi (t+\ell \left( t\right) )}}{%
\varphi ^{\prime }\left( t_{0}+\ell \left( t_{0}\right) \right) e^{-\ln
\left\vert \gamma _{\eta }\right\vert \varphi (t_{0}-\ell \left(
t_{0}\right) )}}E_{\ell }\left( t_{0}\right) \leq E_{\ell }\left( t\right)
\label{stab3} \\
\leq \frac{\varphi ^{\prime }\left( t+\ell \left( t\right) \right) e^{-\ln
\left\vert \gamma _{\eta }\right\vert \varphi (t-\ell \left( t\right) )}}{%
\varphi ^{\prime }\left( t_{0}-\ell \left( t_{0}\right) \right) e^{-\ln
\left\vert \gamma _{\eta }\right\vert \varphi (t_{0}+\ell \left(
t_{0}\right) )}}E_{\ell }\left( t_{0}\right) ,\text{ \ \ for }t\in \lbrack
t_{0},t_{1}].
\end{multline}

\item If $0<\ell ^{\prime }\left( t\right) <1$ on $[t_{0},t_{1}]$, then $%
\varphi ^{\prime }$and $E_{\ell }\left( t\right) $ are nonincreasing and%
\begin{multline}
\frac{\varphi ^{\prime }\left( t+\ell \left( t\right) \right) e^{-\ln
\left\vert \gamma _{\eta }\right\vert \varphi (t+\ell \left( t\right) )}}{%
\varphi ^{\prime }\left( t_{0}-\ell \left( t_{0}\right) \right) e^{-\ln
\left\vert \gamma _{\eta }\right\vert \varphi (t_{0}-\ell \left(
t_{0}\right) )}}E_{\ell }\left( t_{0}\right) \leq E_{\ell }\left( t\right)
\label{stab4} \\
\leq \frac{\varphi ^{\prime }\left( t-\ell \left( t\right) \right) e^{-\ln
\left\vert \gamma _{\eta }\right\vert \varphi (t-\ell \left( t\right) )}}{%
\varphi ^{\prime }\left( t_{0}+\ell \left( t_{0}\right) \right) e^{-\ln
\left\vert \gamma _{\eta }\right\vert \varphi (t_{0}+\ell \left(
t_{0}\right) )}}E_{\ell }\left( t_{0}\right) ,\text{ \ \ for }t\in \lbrack
t_{0},t_{1}].
\end{multline}
\end{itemize}

\begin{proof}
Since (\ref{stab}) holds also for $t=t_{0}$, then the corollary follows by
combining the inequalities%
\begin{equation*}
\frac{e^{\ln \left\vert \gamma _{\eta }\right\vert \varphi (t-\ell \left(
t\right) )}}{M\left( t\right) }E_{\ell }\left( t\right) \leq \mathcal{S}%
_{\eta }\leq \frac{e^{\ln \left\vert \gamma _{\eta }\right\vert \varphi
(t_{0}+\ell \left( t_{0}\right) )}}{m\left( t_{0}\right) }E_{\ell }\left(
t_{0}\right)
\end{equation*}%
and%
\begin{equation*}
\frac{e^{\ln \left\vert \gamma _{\eta }\right\vert \varphi (t_{0}-\ell
\left( t_{0}\right) )}}{M\left( t_{0}\right) }E_{\ell }\left( t_{0}\right)
\leq \mathcal{S}_{\eta }\leq \frac{e^{\ln \left\vert \gamma _{\eta
}\right\vert \varphi (t+\ell \left( t\right) )}}{m\left( t\right) }E_{\ell
}\left( t\right) .
\end{equation*}
\end{proof}

\section*{Acknowledgements} The authors have been supported by the General
Direction of Scientific Research and Technological Development (Algerian
Ministry of Higher Education and Scientific Research) PRFU \#
C00L03UN280120220010.

\subsection*{ORCID} Abdelmouhcene Sengouga https://orcid.org/0000-0003-3183-7973.


\begin{thebibliography}{99} 

\bibitem{AmBE2018}
K.~Ammari, A.~Bchatnia, and K.~{El Mufti}.
\newblock Stabilization of the wave equation with moving boundary.
\newblock \emph{Eur. J. Control}, 39:\penalty0 35--38, 2018.
\newblock ISSN 0947-3580. 
 

\bibitem{BaCh1981}
C.~{Bardos} and G.~{Chen}.
\newblock {Control and stabilization for the wave equation. III: Domain with
  moving boundary.}
\newblock \emph{{SIAM J. Control Optim.}}, 19:\penalty0 123--138, 1981.
 

\bibitem{Cher1994}
M.~Cherkaoui.
\newblock Estimation optimale du taux de d{\'e}croissance de l'{\'e}nergie pour
  une {\'e}quation des ondes avec contr{\^o}le fronti{\`e}re.
\newblock \emph{Rapport de recherche N°2328 - INRIA}, 1994.

\bibitem{CoZu1995}
S.~Cox and E.~Zuazua.
\newblock The rate at which energy decays in a string damped at one end.
\newblock \emph{Indiana Univ. Math. J.}, pages 545--573, 1995.
 
\bibitem{GhSe2022a}
S.~E. Ghenimi and A.~Sengouga.
\newblock Energy decay estimates for an axially travelling string damped at one
  end.
\newblock \emph{Submitted}.
\newblock URL \url{http://arxiv.org/abs/2211.10537}.
\newblock 14 pages.

\bibitem{Guga2008}
M.~Gugat.
\newblock Optimal boundary feedback stabilization of a string with moving
  boundary.
\newblock \emph{IMA J. Math. Control Inform.}, 25\penalty0 (1):\penalty0
  111--121, 2008.

\bibitem{HaHo2019}
B.~H. Haak and D.-T. Hoang.
\newblock Exact observability of a 1-dimensional wave equation on a
  noncylindrical domain.
\newblock \emph{SIAM J. Control Optim.}, 57\penalty0 (1):\penalty0 570--589,
  2019.
 

\bibitem{LuFe2020}
L.~Lu and Y.~Feng.
\newblock Observability and stabilization of $1-d$ wave equations with moving
  boundary feedback.
\newblock \emph{Acta Appl. Math.}, 170\penalty0 (1):\penalty0 731--753, 2020.
 

\bibitem{Mokh2022}
Y.~Mokhtari.
\newblock Boundary controllability and boundary time-varying feedback
  stabilization of the 1d wave equation in non-cylindrical domains.
\newblock \emph{Evol. Equ. Control Theory}, 11\penalty0 (2):\penalty0 373--397,
  2022.
 

\bibitem{Moor1970}
G.~T. Moore.
\newblock Quantum theory of electromagnetic field in a variable-length
  one-dimensional cavity.
\newblock \emph{J. Math. Phys.}, 11\penalty0 (9):\penalty0 2679--2691, 1970.
 
 \bibitem{QuRu1977}
J.~P. Quinn and D.~L. Russell.
\newblock Asymptotic stability and energy decay rates for solutions of
  hyperbolic equations with boundary damping.
\newblock \emph{Proc. R. Soc. Edinb. A: Math.}, 77\penalty0 (1-2):\penalty0
  97–127, 1977.
 

\bibitem{Seng2018}
A.~Sengouga.
\newblock {Observability of the 1-D wave equation with mixed boundary
  conditions in a non-cylindrical domain}.
\newblock \emph{Mediterr. J. Math.}, 15\penalty0 (62):\penalty0 1--22,
  2018. 

\bibitem{Seng2018a}
A.~Sengouga.
\newblock {Observability and controllability of the 1-D wave cquation in
  domains with moving boundary}.
\newblock \emph{Acta Appli. Math.}, 157:\penalty0 117--128, 2018. 

\bibitem{SuLL2015}
H.~Sun, H.~Li, and L.~Lu.
\newblock Exact controllability for a string equation in domains with moving
  boundary in one dimension.
\newblock \emph{Electron. J. Diff. Equations}, 2015\penalty0 (98):\penalty0
  1--7, 2015.

\bibitem{Vese1988}
K.~Veselić.
\newblock On linear vibrational systems with one dimensional damping.
\newblock \emph{Appl. Anal.}, 29\penalty0 (1-2):\penalty0 1--18, 1988. 

\bibitem{Vesn1971}
A.~I. Vesnitskii.
\newblock The inverse problem for a one-dimensional resonator the dimensions of
  which vary with time.
\newblock \emph{Radiophysics and Quantum Electronics}, 14\penalty0
  (10):\penalty0 1209--1215, 1971. 

\bibitem{VePo1975}
A.~I. Vesnitskii and A.~I. Potapov.
\newblock Some general properties of wave processes in one-dimensional
  mechanical systems of variable length.
\newblock \emph{Sov. App. Mechanics}, 11\penalty0 (4):\penalty0 422--426, 1975.
\newblock ISSN 1573-8582. 

\end{thebibliography}
\end{document}